\documentclass[1 [leqno,11pt]{amsart}
\usepackage{amssymb, amsmath,amsmath,latexsym,amssymb,amsfonts,amsbsy, amsthm}
\def\inte#1{
\displaystyle\mathop{#1\kern0pt}^\circ }

\def\eqdefa{\buildrel\hbox{\footnotesize def}\over =}

\newcommand{\beqo}{\begin{equation*}}
\newcommand{\eeqo}{\end{equation*}}
\newcommand{\beno}{\begin{eqnarray*}}
\newcommand{\eeno}{\end{eqnarray*}}
\numberwithin{equation}{section}
\let\pa=\partial
\let\al=\alpha
\let\b=\beta

\let\r=\rho

\let\f=\frac

\let\p=\psi

\let\D=\Delta

\let\Om=\Omega

\let\tri=\triangle
\let\ep=\varepsilon
\def\cC{{\mathcal C}}

\def\cH{{\mathcal H}}

\def\cM{{\mathcal M}}

\def\cR{{\mathcal R}}

\def\virgp{\raise 2pt\hbox{,}}
\def\cdotpv{\raise 2pt\hbox{;}}

\def\eqdefa{\buildrel\hbox{\footnotesize def}\over =}

\def\C{\mathop{\mathbb C\kern 0pt}\nolimits}
\def\DD{\mathop{\mathbb D\kern 0pt}\nolimits}
\def\EE{\mathop{{\mathbb E \kern 0pt}}\nolimits}
\def\K{\mathop{\mathbb K\kern 0pt}\nolimits}
\def\N{\mathop{\mathbb N\kern 0pt}\nolimits}
\def\Q{\mathop{\mathbb Q\kern 0pt}\nolimits}
\def\R{\mathop{\mathbb R\kern 0pt}\nolimits}
\def\SS{\mathop{\mathbb S\kern 0pt}\nolimits}
\def\ZZ{\mathop{\mathbb Z\kern 0pt}\nolimits}
\def\T{\mathop{\mathbb T\kern 0pt}\nolimits}
\def\P{\mathop{\mathbb P\kern 0pt}\nolimits}
\def\I{\mathop{\mathbb I\kern 0pt}\nolimits}

\def\dive{\mathop{\rm div}\nolimits}

\def\no{\noindent}
\def\na{\nabla}
\def\p{\partial}

\newcommand{\beq}{\begin{equation}}
\newcommand{\eeq}{\end{equation}}
\newcommand{\ben}{\begin{eqnarray}}
\newcommand{\een}{\end{eqnarray}}
\newcommand{\andf}{\quad\hbox{and}\quad}

\newtheorem{defi}{Definition}[section]
\newtheorem{thm}{Theorem}[section]
\newtheorem{lem}{Lemma}[section]
\newtheorem{rmk}{Remark}[section]

\newtheorem{prop}{Proposition}[section]

\begin{document}

\title[Decay estimates of incompressible inhomogeneous NS equations]
{ Decay estimates of global solution to 2D incompressible
inhomogeneous Navier-Stokes equations with variable viscosity }
\author[J. Huang]{Jingchi Huang}\address[J. HUANG]
{Academy of Mathematics $\&$ Systems Science, Chinese Academy of
Sciences, Beijing 100190, P. R. CHINA} \email{jchuang@amss.ac.cn}
\author[M. PAICU]{Marius Paicu}
\address [M. PAICU]
{Universit\'e  Bordeaux 1\\
 Institut de Math\'ematiques de Bordeaux\\
F-33405 Talence Cedex, France}
\email{marius.paicu@math.u-bordeaux1.fr}
\date{20/Nov/2012}
\maketitle
\begin{abstract}
In this paper, we investigate the time decay behavior to Lions weak
solution of 2D incompressible inhomogeneous Navier-Stokes equations.
\end{abstract}


\vskip 0.2cm


\setcounter{equation}{0}
\section{Introduction}

The main purpose of this paper is to investigate the decay estimates
for the global solutions of the following two-dimensional
incompressible inhomogeneous Navier-Stokes equations with viscous
coefficient depending on the density \beq\label{INS} \left\{
\begin{array}
{l} \displaystyle \pa_t\rho + \dive(\rho u) = 0, \qquad (t,x)\in\Bbb{R}^+\times\Bbb{R}^2,\\
\displaystyle \pa_t (\rho u) + \dive(\rho u \otimes u) - \dive(\mu(\rho)\cM(u)) + \nabla \Pi = 0, \\
\displaystyle \dive u = 0,
\end{array}
\right. \eeq where $\rho,$ $u=(u_1, u_2)$ stand for the density and
velocity of the fluid respectively, $\cM(u)=\na u +\na^T u,$ $\Pi$
is a scalar pressure function, and in general, the viscosity
coefficient $\mu (\rho)$ is a smooth positive function on
$[0,\infty).$ Such system describes a fluid which is obtained by
mixing two immiscible fluids that are incompressible and that have
different densities. It may also describe a fluid containing a
melted substance. One may check \cite{pl} for the detailed
derivation.

When $\mu(\rho)$ is independent of $\rho,$ i.e. $\mu$ is a positive
constant, and the initial density has a positive lower bound, Lady\v
zenskaja and Solonnikov  \cite{LS} first addressed the question of
unique solvability of \eqref{INS}. More precisely, they considered
the system \eqref{INS} in a bounded domain $\Om$ with homogeneous
Dirichlet boundary condition for $u.$ Under the assumption that
$u_0\in W^{2-\frac2p,p}(\Om)$ $(p>d)$ is divergence free and
vanishes on  $\p\Om$ and that $\r_0\in C^1(\Om)$ is bounded away
from zero, then they \cite{LS} proved
\begin{itemize}
\item Global well-posedness in dimension $d=2;$
\item Local well-posedness in dimension $d=3.$ If in addition $u_0$ is small in $W^{2-\frac2p,p}(\Om),$
then global well-posedness holds true.
\end{itemize}
Danchin \cite{danchin} proved similar well-posedness result of
\eqref{INS} in the whole space case and the initial data in the
almost critical spaces. In particular, in two dimension, he proved
the global well-posedness of \eqref{INS} provided the initial data
$(\rho_0, u_0)$ satisfying $\rho_0-1\in H^{1+\al}(\R^2),$
$\rho_0\geq m>0,$ and $u_0\in H^\b(\R^2)$ for any $\al\in(0,1)$ and
$\b\in (0,1].$

In general, Lions \cite{pl} (see also the references therein) proved
the global existence of weak solutions to \eqref{INS} with finite
energy. Yet the uniqueness and regularities of such weak solutions
are big open questions even in two space dimensions. Except under
the additional assumptions that \beq\label{desjassume}
\|\mu(\rho_0)-1\|_{L^\infty(\T^2)}\leq \ep \andf u_0\in H^1(\T^2),
\eeq Desjardins \cite{desj} proved the following theorem.
\begin{thm}\label{desjthm}
{\sl Let $\rho_0 \in L^\infty(\T^2)$ and $\dive u_0=0.$ Then there
exists $\ep>0$ such that under the assumption \eqref{desjassume},
Lions weak solutions (\cite{pl}) to \eqref{INS} satisfy the
following regularity properties hold for all $T>0:$
\begin{enumerate}
\item
$u\in L^\infty((0,T); H^1(\T^2))$ and $\sqrt{\rho}\pa_t u\in
L^2((0,T)\times \T^2) ,$
\item
$\rho$ and $\mu(\rho)\in L^\infty((0,T)\times \T^2)\cap
C([0,T];L^p(\T^2))$ for all $p\in [1,\infty),$
\item
$\na (\Pi-\cR_i\cR_j(\mu \cM(u)_{ij}))$ and $\na (\P\otimes\Q(\mu
\cM(u)_{ij}))\in L^2((0,T)\times \T^2),$
\item
$\Pi$ may be renormalized in such a way that for some universal
constant $C_0>0,$ $\Pi$ and $\na u \in L^2((0,T); L^p(\T^2))$ for
all $p\in [4,p^*),$ where
$\f1{p^*}=2C_0\|\mu(\rho_0)-1\|_{L^\infty}.$
\end{enumerate}In which, we denote $\cR$ as the Riesz transform: $\cR
=\na \tri^{-\f12}.$ $\Q=\na \tri^{-1}\dive$ and $\P=\I-\Q$
respectively denote the projection on the space of curl-free and
divergence-free vector fields.}
\end{thm}
In order to investigate the global well-posedness of thus solutions,
we first need to study the global-in-time type estimates. However,
because of the difficulties of the continuity equation in
\eqref{INS} being of hyperbolic nature and the estimate of the
diffusion term in the momentum equation, we shall first study the
time decay of the solutions, which is very much motivated by
\cite{gz, sch, wie}.

\begin{thm}\label{mainthm1} {\sl For $1<p<2,$ let $u_0 \in
L^p(\R^2)\cap H^1(\R^2),$ $\rho_0-1\in L^2(\R^2)$ and $\rho_0\in
L^\infty(\R^2)$ with a positive lower bound. We assume that $(\rho,
u, \na p)$ is a given Lions weak solution of \eqref{INS} with
initial data $(\rho_0, u_0).$ Denote $\mu(1)=\mu_0,$ then under the
assumption \beq\label{smallassume1}
\|\mu(\rho)-\mu_0\|_{L^\infty(\R^+;L^\infty(\R^2))}\leq \ep_0, \eeq
for a small constant $\ep_0,$ there exists a constant $C_1$ which
depends on $\|\rho_0-1\|_{L^2},$ $\|u_0\|_{L^p}$ and $\|u_0\|_{H^1}$
such that there hold \beq\label{timedecay1} \|u(t)\|_{L^2}^2\leq
C_1(t+e)^{-2\b(p)}, \qquad \|\na u(t)\|_{L^2}^2\leq C_1
(t+e)^{-1-2\b(p)+\ep},\eeq \beq\label{timedecay2} \int_0^\infty
\|u_t\|_{L^2} + \|\P\dive\bigl(\mu(\rho)\cM(u)\bigr)\|_{L^2}
+\|\Q\dive\bigl(\mu(\rho)\cM(u)\bigr)-\na\Pi\|_{L^2}\,dt \leq C_1,
\eeq \beq\label{timedecay3} \int_0^\infty
(t+e)^{1+2\b(p)-\ep}\Bigl(\|u_t\|_{L^2} +
\|\P\dive\bigl(\mu(\rho)\cM(u)\bigr)\|_{L^2}
+\|\Q\dive\bigl(\mu(\rho)\cM(u)\bigr)-\na\Pi\|_{L^2}\Bigr)^2\,dt
\leq C_1, \eeq with $\b(p)=\f12(\f2p-1)$ and any $\ep>0.$}
\end{thm}
\begin{rmk}\label{rmk0.1}
The first estimate of \eqref{timedecay1} coincides with the
$L^2$-norm decay result in \cite{sch, wie} for the weak solutions of
the two-dimensional classical Navier-Stokes system, and also
coincides with the result in \cite{gz} for \eqref{INS}. When
$\mu(\rho)$ be a constant, we can get optimal decay of $\|\na
u\|_{L^2}^2$ with the order $-1-2\b(p),$ see \cite{huang}. Notice
the main ingredients of the proof in \cite{huang, sch, wie} are the
usual energy estimates and the phase space analysis. In our case,
due to the additional difficulties mentioned above, we not only need
to apply phase space analysis, but also need more explicit energy
estimates, see Proposition \ref{energyprop} below. We note also that
the 3D case with constant viscosity was studied in \cite{AGZ1}.
Using energy estimates with weight in time and the Fourier splitting
method of Schonbek \cite{sch} we can generalize this decay in time
estimates to the 3D case with variable viscosity.
\end{rmk}

Motivated by Proposition \ref{energyprop}, we have a more general
result. Indeed, using interpolation argument we obtain a similar
decay rate of the solution, under a weaker assumption on the initial
volocity.

\begin{thm}\label{mainthm2} {\sl For $1<p<2$ and $0<\al<1,$ let $u_0 \in
L^p(\R^2)\cap H^{\al}(\R^2),$ $\rho_0-1\in L^2(\R^2)$ and $\rho_0\in
L^\infty(\R^2)$ with a positive lower bound. We assume that $(\rho,
u, \na p)$ is a given Lions weak solution of \eqref{INS} with
initial data $(\rho_0, u_0).$ Then under the assumption
\eqref{smallassume1}, there exists a constant $C_{\al}$ which
depends on $\|\rho_0-1\|_{L^2},$ $\|u_0\|_{L^p}$ and
$\|u_0\|_{H^{\al}}$ such that there hold \beq\label{timedecay4}
\|u(t)\|_{L^2}^2\leq C_{\al}(t+e)^{-2\b(p)}, \quad \|\na
u(t)\|_{L^2}^2\leq C_{\al}(t+e)^{-1-2\b(p)+\ep},\eeq
\beq\label{timedecay5} \int_0^\infty \|u_t\|_{L^2} +
\|\P\dive\bigl(\mu(\rho)\cM(u)\bigr)\|_{L^2}
+\|\Q\dive\bigl(\mu(\rho)\cM(u)\bigr)-\na\Pi\|_{L^2}\,dt \leq
C_{\al}, \eeq \beq\label{timedecay6} \int_0^\infty
t^{1-r}(t+e)^{r+2\b(p)-\ep}\Bigl(\|u_t\|_{L^2} +
\|\P\dive\bigl(\mu(\rho)\cM(u)\bigr)\|_{L^2}
+\|\Q\dive\bigl(\mu(\rho)\cM(u)\bigr)-\na\Pi\|_{L^2}\Bigr)^2\,dt
\leq C_{\al}, \eeq with any $\ep>0$ and $0<r<\al.$}
\end{thm}
\begin{rmk}\label{rmk0.2}
We note also that the 3D case with constant viscosity was studied in
\cite{AGZ1}. Using energy estimates with weight in time and the
Fourier splitting method of Schonbek \cite{sch} we can generalize
this decay in time estimates to the 3D case with variable viscosity.
\end{rmk}

In the second part of this paper, we investigate the regularity
propagation of transport equation. We consider the transport
equation: \beq\label{transport} \left\{
\begin{array}
{l} \displaystyle \pa_t\rho + u \na \rho = 0, \qquad (t,x)\in\Bbb{R}^+\times\Bbb{R}^2,\\
\displaystyle \dive u = 0,\\
\displaystyle \rho|_{t=0} =\rho_0.
\end{array}
\right. \eeq In the case of $u\in L^1(Lip),$ for any small positive
regularity, it is well known that \beno
\|\rho(t)\|_{B^{\ep}_{p,r}}\leq
\|\rho_0\|_{B^{\ep}_{p,r}}\exp(\|u\|_{L^1_t(Lip)}). \eeno And if the
regularity index is $0$, follows from \cite{v}, we have \beno
\|\rho(t)\|_{B^{0}_{p,r}}\leq
C\|\rho_0\|_{B^{0}_{p,r}}(1+\|u\|_{L^1_t(Lip)}). \eeno We want to
know how it changes from zero regularity to positive regularity. So
we define a Besov space with logarithms regularity $B^{\eta
\ln}_{\infty,1}$, which is just between zero regularity and positive
regularity, see Definition \ref{def1.2} and Remark \ref{rmk1.1}
below. So we gain a polynomial relation between the velocity and the
density, which is the case between exponential and linear cases.

According these two results, we give an application about global
existence to solutions of \eqref{INS}.
\begin{thm}\label{mainthm3}
{\sl For $1<p<\f43,$ let $u_0\in L^p(\R^2)\cap H^1(\R^2).$ Let
$\rho_0-1\in B^{1+\ep}_{2,1}(\R^2)$ for any $\ep>0,$ and $\rho_0\in
L^\infty(\R^2)$ with positive lower bound. Then there exist positive
constant $\eta>1$ and $C_0,$ $c_0$ such that if
\beq\label{smallassume2}
\|\mu(\rho_0)-\mu_0\|_{B^{(\eta+1)\ln}_{\infty, 1}}
\Bigl(\f{(1+\mu_0)G(\rho_0,u_0)}{\mu_0}\Bigr)^{\eta+1}\exp\big\{(\eta+1)\exp(C_0\|u_0\|_{L^2}^4)\big\}
\leq c_0\mu_0, \eeq where \beq\label{constant1}
\begin{split}
&G(\rho_0,u_0)=G_1\exp(G_2),\\
&G_1=\|\rho_0-1\|_{L^2}+\|\rho_0-1\|_{L^2}^7 + \|u_0\|_{L^p}
+\|u_0\|_{L^p}^7 +\|u_0\|_{H^1} + \|u_0\|_{L^p}^7\|u_0\|_{H^1}^7\\
&\qquad+(1+ \|\rho_0-1\|_{L^2}^7)\|u_0\|_{H^1}^{14},\\
&G_2= \|u_0\|_{L^p}^4+
\|u_0\|_{H^1}^4+\|\rho_0-1\|_{L^2}^4+\|u_0\|_{L^p}^4\|u_0\|_{H^1}^4+
(1+\|\rho_0-1\|_{L^2}^4)\|u_0\|_{H^1}^8,
\end{split}\eeq \eqref{INS} has a global solution $(\rho, u)$ such
that $\rho-1 \in L^\infty((0,T); B^{1+\f{\ep}2}_{2,1}(\R^2))$ for
any $T>0,$ and $\na u \in L^1(\R^+; B^0_{\infty, 2}).$ }
\end{thm}
\begin{rmk}\label{rmk0.3}
We don't mention the result of uniqueness, and this is well known in
\cite{danchin}.
\end{rmk}

\no{\bf The organization of the paper.} In the second section, we
collect some basic facts on  Littlewood-Paley theory and integral
inequalities, which have been used throughout this paper. In Section
3, we shall present the proof of Theorem \ref{mainthm1}. In Section
4, we shall prove Theorem \ref{mainthm2}. In Section 5, we give an
application of Theorem \ref{mainthm1}.

Let us complete this section by the notations we shall use in this
context:

\no{\bf Notation.} For $a\lesssim b$, we mean that there is a
uniform constant $C,$ which may be different on different lines,
such that $a\leq Cb$. $a\approx b$ means that there is two positive
uniform constant $c,$ $C$ such that $cb\leq a\leq Cb.$ We shall
denote by $(c_{j,r})_{j\in\N}$ to be a generic element of
$\ell^r(\N)$ so that $c_{j,r}\geq 0$ and $\sum_{j\in\N}c_{j,r}^r=1.$

\setcounter{equation}{0}
\section{Preliminaries}
First, we are going to recall some facts on the Littlewood-Paley
Theory, one may check \cite{BCD} for details. Let $B
\overset{def}{=} \{ \xi \in \R^2, |\xi| \leq \frac{4}{3} \}$ and
$\mathcal C \overset{def}{=} \{ \xi \in \R^2, \frac{3}{4} \leq |\xi|
\leq \frac{8}{3} \}$. Let $\chi \in \cC_c^{\infty}(B)$ and $\varphi
\in \mathcal C_c^{\infty}(\cC)$ which satisfy \beqo \chi(\xi) +
\sum_{j\geq 0} \varphi(2^{-j}\xi) = 1,\quad \xi \in \R^2, \eeqo
we denote $h \eqdefa \mathcal F^{-1}\varphi$ and $\tilde{h} \eqdefa \mathcal F^{-1}\chi$. Then the Littlewood-Paley operators $\Delta_j$ and $S_j$ can be defined as follows\\
\ben
&&\displaystyle \Delta_jf \eqdefa \varphi(2^{-j}D)f = 2^{2j}\int_{\R^2} h(2^jy)f(x-y)\,dy,\quad \hbox{for}\, j \geq 0,\nonumber\\
&&\displaystyle S_jf \eqdefa \chi(2^{-j}D)f = \sum\limits_{-1\leq k \leq j-1}\Delta_kf = 2^{2j}\int_{\R^2}\tilde{h}(2^jy)f(x-y)\,dy,\\
&&\displaystyle \Delta_{-1}f \eqdefa S_{-1}f, \qquad
S_{-2}f=0.\nonumber \een With the introduction of $\Delta_j$ and
$S_j,$ we define two norm which will be used throughout of our work.
\begin{defi}\label{def1.1}
{\sl Let $s\in \R$ and $1\leq p,r \leq \infty.$ The inhomogeneous
Besov space $B^s_{p,r}$ consists of all tempered distributions $u$
such that
$$\|u\|_{B^s_{p,r}} \eqdefa \|
(2^{js}\|\D_j u\|_{L^p})_{j\geq -1}\|_{\ell^r} <\infty.$$ }
\end{defi}

\begin{defi}\label{def1.2}
{\sl Let $\eta >0.$ The logarithms inhomogeneous Besov space
$B^{\eta \ln}_{\infty,1}$ consists of all tempered distributions $u$
such that
$$\|u\|_{B^{\eta \ln}_{\infty,1}} \eqdefa \sum_{j\geq -1}
(2+j)^{\eta}\|\D_j u\|_{L^\infty} <\infty.$$ }
\end{defi}
\begin{rmk}\label{rmk1.1}
One may see that for any positive $\ep$ and $\eta,$
$B^{\ep}_{\infty, 1}\subset B^{\eta \ln}_{\infty,1} \subset
B^0_{\infty,1}.$
\end{rmk}

Let us recall the following lemmas from \cite{BCD}.

\begin{lem}[Bernstein's inequality]\label{ber}
{\sl Let $1 \leq p \leq q \leq \infty.$ Assume that $f \in L^p,$
then there exists a positive constant C independent of $f,$ $j$ such
that \beno
\hbox{supp} \hat{f} \subset \{|\xi| \lesssim 2^j\} \Rightarrow \|\pa^{\alpha}f\|_{L^q} \leq C2^{j|\alpha|+2j(\frac{1}{p}-\frac{1}{q})}\|f\|_{L^p},\\
\hbox{supp} \hat{f} \subset \{|\xi| \approx 2^j\} \Rightarrow
\|f\|_{L^p} \leq C2^{-j|\alpha|}\|\pa^{\alpha}f\|_{L^p}. \eeno }
\end{lem}

\begin{lem}
{\sl Let $\phi$ be a smooth function supported in the annulus $\{\xi
\in \R^2: |\xi| \approx 1\}.$ Then, there exist two positive
constants $c$ and $C$ depending only on $\phi$ such that for any $1
\leq p \leq \infty$ and $\lambda > 0,$ we have \beqo \|\phi
(\lambda^{-1}D)e^{t\Delta}f\|_{L^p} \leq
Ce^{-ct\lambda^2}\|\phi(\lambda^{-1}D)f\|_{L^p}. \eeqo }
\end{lem}

In what follows, we will constantly use Bony's decomposition
\begin{equation}\label{bony}
uv = T_u v + T_v u + R(u,v)
\end{equation}
where \beqo T_uv = \sum\limits_{j \geq -1}S_{j-1}u\Delta_jv \quad
\hbox{and} \quad R(u,v) = \sum\limits_{j \geq -1}\Delta_ju
\tilde{\Delta}_jv, \eeqo where $\tilde{\Delta}_j =
\sum_{i=-1}^{1}\Delta_{j+i}.$

Also, we need some calculus inequalities which can be found in
\cite{wie}.
\begin{lem}
{\sl Let $m \in \R^+,$ $0\leq \al <1$ and $\beta > 0.$ Then
\begin{enumerate}
\item
$\int_0^t (s+e)^{-1} \ln (s+e)^{-m}\,ds \leq \frac{1}{m-1} \quad
\hbox{for} \quad m >1,$
\item
there is some $\gamma_m>0$ such that $\int_0^t (s+e)^{-1-\beta}\ln
(s+e)^m \,ds \leq \gamma_m\beta^{-(m+1)},$
\item
there is some $\gamma_{m,\al} > 0$ such that, for all $t>0,$\\
$\int_0^t (s+e)^{-\al}\ln (s+e)^{-m}\,ds \leq
\gamma_{m,\al}(t+e)^{1-\al}\ln(t+e)^{-m}.$
\end{enumerate}}
\end{lem}

Finally, we need the integral form of Gronwall's inequality, which
is well known.
\begin{prop}
{\sl Let $f,g,h$ be positive functions defined on $\R^+,$ and
$h(t)\in L^1_{loc}.$ If
$$f(t)\leq g(t)+\int_0^th(s)f(s)\,ds,$$
then following estimate holds: \beq\label{gronwall} f(t)\leq
g(t)+\int_0^t h(s)g(s)\exp(\int_s^th(\tau)\,d\tau)\,ds. \eeq }
\end{prop}

\setcounter{equation}{0}
\section{The Proof of Theorem \ref{mainthm1}}

In this section, we will prove Theorem \ref{mainthm1}. First, we
have some energy estimates.
\begin{prop}\label{energyprop}
{\sl Let $v\in L^\infty(\R^+;L^2)\cap L^2(\R^+;\dot{H}^1),$ $\dive
v=0.$ Assume that $u_0 \in H^1(\R^2)$ and $\rho_0 \in
L^\infty(\R^2)$ with positive lower bound. $f(t)$ be a positive
second-order differentiable function satisfies $f'(t)\geq 0$ and
$f''(t)\geq 0.$ $(\rho, u)$ be the global weak solution of the
linear system: \beq\label{LINS} \left\{
\begin{array}
{l} \displaystyle \pa_t\rho + v\na \rho = 0, \\
\displaystyle \rho\pa_t u + \rho v\na u - \dive(\mu(\rho)\cM(u)) + \nabla \Pi = 0, \\
\displaystyle \dive u = 0, \\
\displaystyle (\rho, u)|_{t=0}=(\rho_0,u_0).
\end{array}
\right. \eeq Then under the assumption \eqref{smallassume1}, we have
the following estimates: \ben\label{ener1}
&&\sup_{0<t<\infty}f(t)\int_{\R^2}\mu(\rho)|\na u|^2(t)\,dx\notag\\
&&\quad +\int_0^\infty f(t)\int_{\R^2}|\sqrt{\rho}u_t|^2+|\P\dive\bigl(\mu(\rho)\cM(u)\bigr)|^2+ |\Q\dive\bigl(\mu(\rho)\cM(u)\bigr)-\na\Pi|^2\,dxdt \notag\\
&&\leq C(f(0)\|\na u_0\|_{L^2}^2 + \int_0^\infty
f'(t)\int_{\R^2}\mu(\rho)|\na u|^2\,dxdt
)\exp\{C(1+\|v\|_{L^\infty(L^2)}^2)\|\na v\|_{L^2(L^2)}^2\}, \een
\beq\label{ener2} \sup_{0<t<\infty}f'(t)\int_{\R^2}\rho|u|^2(t)\,dx
+ \int_0^\infty f'(t)\int_{\R^2}\mu(\rho)|\na u|^2\,dxdt \leq C(
f'(0)\|u_0\|_{L^2}^2 + \int_0^\infty
f''(t)\int_{\R^2}\rho|u|^2\,dxdt). \eeq }
\end{prop}

\begin{proof} First, we follow the line of the proof of Theorem \ref{desjthm}, see \cite{desj}. By taking $L^2$
inner product of the momentum equation of \eqref{LINS} with
$f(t)u_t$ and using integration by parts, we deduce that \beno
f(t)\int_{\R^2}|\sqrt{\rho}u_t|^2\,dx + f(t)\int_{\R^2} (\rho v \na
u)\cdot u_t\,dx + f(t)\int_{\R^2} \mu(\rho)\na u :\na u_t\,dx=0.
\eeno Note that \beno f(t)\int_{\R^2} \mu(\rho)\na u
:\na u_t\,dx &=& \f12 \pa_t [f(t)\int_{\R^2}\mu(\rho)|\na u|^2\,dx]-\f12f'(t)\int_{\R^2}\mu(\rho)|\na u|^2\,dx\\
&&-\f12f(t)\int_{\R^2} \pa_t\mu(\rho)|\na u|^2\,dx, \eeno and from
the derivation of (29) in \cite{desj} that \beno -\int_{\R^2}
\pa_t\mu(\rho)|\na u|^2\,dx&=& \int_{\R^2}\dive(\mu(\rho)v)|\na
u|^2\,dx\\
&=& \int_{\R^2}(v\na)u\cdot\dive(\mu(\rho)\cM(u))\,dx +\int_{\R^2} \mu(\rho)\mbox{tr}(\na v \na u \cM(u))\,dx \\
&=&\int_{\R^2}(v\na) u\cdot(\rho u_t+\rho v\na u +\na
\Pi)\,dx+\int_{\R^2} \mu(\rho)\mbox{tr}(\na v\na u \cM(u))\,dx,\eeno
we have \beno &&\f{d}{dt} [f(t)\int_{\R^2}\mu(\rho)|\na u|^2\,dx] +
f(t)\int_{\R^2}|\sqrt{\rho}u_t|^2\,dx\\
&\lesssim& f'(t)\int_{\R^2}\mu(\rho)|\na u|^2\,dx +
f(t)\int_{\R^2}|\sqrt{\rho}v\na u|^2\,dx\\
&&+f(t)\int_{\R^2}\mu(\rho)|\na v||\na u|^2\,dx +f(t)
\Big|\int_{\R^2}\Pi\pa_i v_j\pa_j u_i\,dx \Big|.\eeno Recall that
\beqo -\mu_0\tri u = \dive\bigl((\mu(\rho)-\mu_0)\cM(u)\bigr)-
\dive\bigl(\mu(\rho)\cM(u)\bigr), \eeqo so that we have \beqo
\mu_0\pa_i u_j =
\cR_i\P_j\cR\bigl((\mu(\rho)-\mu_0)\cM(u)\bigr)-\cR_i\P_j\cR\bigl(\mu(\rho)\cM(u)\bigr).
\eeqo Estimating it in the $L^4(\R^2)$ and using the
Gagliardo-Nirenberg inequality, we can write \beno \|\na u\|_{L^4}
&\lesssim& \|\P\otimes\Q\bigl((\mu(\rho)-\mu_0)\cM(u)\bigr)\|_{L^4}
+ \|\P\otimes\Q\bigl(\mu(\rho)\cM(u)\bigr)\|_{L^4}\\
&\lesssim& \|\mu(\rho)-\mu_0\|_{L^\infty(\R^+;L^\infty)}\|\na
u\|_{L^4}+
\|\P\otimes\Q\bigl(\mu(\rho)\cM(u)\bigr)\|_{L^2}^{\f12}\|\na\Bigl(\P\otimes\Q\bigl(\mu(\rho)\cM(u)\bigr)\Bigr)\|_{L^2}^{\f12}
\eeno Finally, using \eqref{smallassume1} and the conservation of
the momentum, we obtain that \beq\label{gradu} \|\na
u\|_{L^4}\lesssim \|\na u\|_{L^2}^{\f12}\|\P(\rho u_t +\rho v\na
u)\|_{L^2}^{\f12},\eeq

Now letting $(-\tri)^{-\f12}\cR$ operate on the equation of
momentum, we get that \beqo \Pi= \cR_i\cR_j\bigl( \mu(\rho)(\pa_i
u_j +\pa_j u_i)\bigr) +(-\tri)^{-\f12}\cR(\rho u_t+\rho v\na
u).\eeqo It follows that \beno\|\Pi
-\cR_i\cR_j(\mu(\rho)\cM(u))\|_{BMO}\lesssim\|\na(\Pi
-\cR_i\cR_j(\mu(\rho)\cM(u)))\|_{L^2}\lesssim \|\rho u_t +\rho v\na
u\|_{L^2}. \eeno We obtain that \beno \Big|\int_{\R^2}\Pi\pa_i
v_j\pa_j u_i\,dx \Big|&\leq& \|\na v\|_{L^2}\|\na u\|_{L^4}^2 +
\|\Pi-\cR_i\cR_j(\mu(\rho)\cM(u))\|_{BMO}\|\pa_i v_j\pa_j u_i\|_{\cH^1}\\
&\leq&\|\na v\|_{L^2}\|\na u\|_{L^2}\|\rho u_t +\rho v\na u\|_{L^2},
\eeno so that \beno f(t)\Big|\int_{\R^2}\Pi\pa_i u_j\pa_j u_i\,dx
\Big|&\leq&C_\ep f(t)\|\na v\|_{L^2}^2\|\na u\|_{L^2}^2 +\ep
f(t)(\|\sqrt{\rho}u_t\|_{L^2}^2 + \|v\na u\|_{L^2}^2). \eeno \beno
f(t)\int_{\R^2}\mu(\rho)|\na v||\na u|^2\,dx &\leq& C f(t)\|\na
v\|_{L^2}\|\na u\|_{L^4}^2\\
&\leq&C_\ep f(t)\|\na v\|_{L^2}^2\|\na u\|_{L^2}^2 +\ep
f(t)(\|\sqrt{\rho}u_t\|_{L^2}^2 + \|v\na u\|_{L^2}^2).\eeno Also
\beno  \|v\na u\|_{L^2}^2 \leq \|v\|_{L^4}^2\|\na u\|_{L^4}^2\leq
\|v\|_{L^2}\|\na v\|_{L^2}\|\na u\|_{L^2}\|\rho u_t
+\rho v\na u\|_{L^2},\\
f(t)\int_{\R^2}|\sqrt{\rho}v\na u|^2\,dx\leq C_\ep
f(t)\|v\|_{L^2}^2\|\na v\|_{L^2}^2\|\na u\|_{L^2}^2 + \ep
f(t)\|\sqrt{\rho}u_t\|_{L^2}^2.\eeno Consequently, \beq\label{ener3}
\begin{split}&\f{d}{dt} [f(t)\int_{\R^2}\mu(\rho)|\na u|^2\,dx] +
f(t)\int_{\R^2}|\sqrt{\rho}u_t|^2\,dx\\
&\lesssim f'(t)\int_{\R^2}\mu(\rho)|\na u|^2\,dx + f(t)\|\na
u\|_{L^2}^2\|\na v\|_{L^2}^2(1+\|v\|_{L^2}^2). \end{split}\eeq
Second, we act the Leray projector $\P$ on the momentum equation of
\eqref{LINS} to get that \beno \P\dive\bigl(\mu(\rho)\cM(u)\bigr)&=&
\P(\rho u_t +\rho v\na u),\\
\Q\dive\bigl(\mu(\rho)\cM(u)\bigr)-\na\Pi&=&\Q(\rho u_t +\rho v\na
u). \eeno Along with \eqref{ener3}, we have
\beno &&\f{d}{dt} [f(t)\int_{\R^2}\mu(\rho)|\na u|^2\,dx]\\
&&\quad +f(t)\int_{\R^2}|\sqrt{\rho}u_t|^2+|\P\dive\bigl(\mu(\rho)\cM(u)\bigr)|^2+|\Q\dive\bigl(\mu(\rho)\cM(u)\bigr)-\na\Pi|^2  \,dx\\
&\leq& C(f'(t)\int_{\R^2}\mu(\rho)|\na u|^2\,dx + f(t)\|\na
u\|_{L^2}^2\|\na v\|_{L^2}^2(1+\|v\|_{L^2}^2)).\eeno Note that $v\in
L^\infty(L^2)\cap L^2(\dot{H}^1),$ so that
$$\int_0^\infty(1+\|v\|_{L^2}^2)\|\na v\|_{L^2}^2\,dt\leq (1+\|v\|_{L^\infty(L^2)}^2)\|\na v\|_{L^2(L^2)}^2,$$
and \eqref{ener1} holds.

The same strategy can be held for $f'(t)u,$ we have \beno
\f12\f{d}{dt}[f'(t)\int_{\R^2}|\sqrt{\rho}u|^2\,dx] +
f'(t)\int_{\R^2}\mu(\rho)|\na u|^2\,dx = \f12
f''(t)\int_{\R^2}|\sqrt{\rho}u|^2\,dx, \eeno so that \beqo
\sup_{0<t<\infty}f'(t)\int_{\R^2}\rho|u|^2(t)\,dx + \int_0^\infty
f'(t)\int_{\R^2}\mu(\rho)|\na u|^2\,dxdt \leq C(
f'(0)\|u_0\|_{L^2}^2 + \int_0^\infty
f''(t)\int_{\R^2}\rho|u|^2\,dxdt). \eeqo
\end{proof}

According these two energy estimates, letting $v=u,$ we can prove
Theorem \ref{mainthm1}. More precisely, we have the following
theorem.
\begin{thm}\label{decaythm}
{\sl Under the assumption of Theorem \ref{mainthm1},
\eqref{timedecay1}, \eqref{timedecay2} and \eqref{timedecay3} hold.
More precisely, we have \beq\label{decay1} \|u(t)\|_{L^2}^2\lesssim
(K+K^7)\exp(K^2)(t+e)^{-2\b(p)}, \qquad \|\na u(t)\|_{L^2}^2\lesssim
(K+K^7)\exp(K^2)(t+e)^{-1-2\b(p)+\ep},\eeq \beq\label{decay2}
\int_0^\infty \|u_t\|_{L^2} +
\|\P\dive\bigl(\mu(\rho)\cM(u)\bigr)\|_{L^2}
+\|\Q\dive\bigl(\mu(\rho)\cM(u)\bigr)-\na\Pi\|_{L^2}\,dt \lesssim
\sqrt{K}+K, \eeq \beq\label{decay3} \int_0^\infty
(t+e)^{1+2\b(p)-\ep}\Bigl(\|u_t\|_{L^2} +
\|\P\dive\bigl(\mu(\rho)\cM(u)\bigr)\|_{L^2}
+\|\Q\dive\bigl(\mu(\rho)\cM(u)\bigr)-\na\Pi\|_{L^2}\Bigr)^2\,dt
\lesssim (K+K^7)\exp(K^2), \eeq where \beq\label{constant2} K
=\Bigl(\|u_0\|_{L^p}^2+
\|u_0\|_{H^1}^2+\|\rho_0-1\|_{L^2}^2+\|u_0\|_{L^p}^2\|u_0\|_{H^1}^2+
(1+\|\rho_0-1\|_{L^2}^2)\|u_0\|_{H^1}^4
\Bigr)\exp\{C\|u_0\|_{L^2}^4\}.\eeq}
\end{thm}
\begin{proof}
We get the standard energy estimate to \eqref{INS} that \beqo
\f{d}{dt}\|\sqrt{\rho}u(t)\|_{L^2}^2 + \|\na u(t)\|_{L^2}^2 \leq 0.
\eeqo Using Schonbek's strategy, we obtain \beq\label{2.1}
\f{d}{dt}\|\sqrt{\rho}u(t)\|_{L^2}^2 +
g^2(t)\|\sqrt{\rho}u(t)\|_{L^2}^2 \leq
Mg^2(t)\int_{S(t)}|\hat{u}(t,\xi)|^2\,\mathrm{d}\xi, \eeq where
$S(t) \overset{def}{=} \{\xi: |\xi|\leq \sqrt{\frac{M}{2}}g(t)\}$
and $g(t)$ satisfying $g(t)\lesssim (1+t)^{-\frac{1}{2}}$. We
rewrite the momentum equation of \eqref{INS} as \beno u(t) =
e^{\mu_0t\tri}u_0 + \int_0^te^{\mu_0(t-s)\tri}\mathbb P \Bigl(
\dive\bigl((\mu(\rho)-\mu_0)\cM(u)\bigr) + (1-\rho)u_t -\rho u\na
u\Bigr)(s)\,ds. \eeno Taking Fourier transform with respect to $x$
variables leads to \beqo |\hat{u}(t,\xi)|\lesssim
e^{-\mu_0t|\xi|^2}|\hat{u_0}(\xi)| +
\int_0^te^{-\mu_0(t-s)|\xi|^2}\bigl[|\xi||\mathcal{F}_x\bigl((\mu(\rho)-\mu_0)\cM(u)\bigr)|
+ |\mathcal{F}_x\bigl((1-\rho)u_t -\rho u\na u\bigr)|\bigr]\,ds,
\eeqo which implies that \ben\label{2.2}
\int_{S(t)}|\hat{u}(t,\xi)|^2\,d\xi &\lesssim& \int_{S(t)}e^{-2\mu_0t|\xi|^2}|\hat{u}_0(\xi)|^2\,d\xi + g^4(t)(\int_0^t\|\mathcal{F}_x\bigl((\mu(\rho)-\mu_0)\cM(u)\bigr)\|_{L^\infty_{\xi}}\,ds)^2 \nonumber \\
&&\qquad + g^2(t)(\int_0^t\|\mathcal{F}_x\bigl((1-\rho)u_t -\rho
u\na u\bigr)\|_{L^\infty_{\xi}}\,ds)^2. \een Note that $u_0 \in L^p$
for $1<p<2,$ one has \beq\label{2.3}
\int_{S(t)}e^{-2\mu_0t|\xi|^2}|\hat{u_0}(\xi)|^2\,d\xi \lesssim
(\int_{S(t)}e^{-2\mu_0qt|\xi|^2}\,d\xi)^{\frac{1}{q}}\|\hat{u_0}(\xi)\|_{L^{p'}}^2
\lesssim \|u_0\|_{L^p}^2(1+t)^{-2\b(p)}, \eeq where
$\frac{1}{q}=\frac{2}{p}-1,$ $\frac{1}{p} + \frac{1}{p'} =1.$

Note that $u\in L^\infty(L^2)\cap L^2(\dot{H}^1)$ and $u_t\in
L^2(L^2),$ we have\beno
(\int_0^t\|\mathcal{F}_x\bigl((\mu(\rho)-\mu_0)\cM(u)\bigr)\|_{L^\infty_{\xi}}\,ds)^2&\leq&
(\int_0^t \|(\mu(\rho)-\mu_0)\cM(u)\|_{L^1} \,ds)^2\\
&\leq&\|\mu(\rho)-\mu_0\|_{L^\infty_t(L^2)}^2(\int_0^t\|\na
u\|_{L^2}\,ds)^2\\
&\leq& C\|\rho_0-1\|_{L^2}^2\|u_0\|_{L^2}^2(1+t), \eeno

\beno (\int_0^t\|\mathcal{F}_x\bigl((1-\rho)u_t
\bigr)\|_{L^\infty_{\xi}}\,ds)^2 \leq
\|1-\rho\|_{L^\infty_t(L^2)}^2(\int_0^t\|u_t\|_{L^2}\,ds)^2\leq
C\|\rho_0-1\|_{L^2}^2\|\na u_0\|_{L^2}^2(1+t), \eeno

\beno (\int_0^t\|\mathcal{F}_x\bigl(\rho u\na
u\bigr)\|_{L^\infty_{\xi}}\,ds)^2 \leq \|\rho
u\|^2_{L^\infty_t(L^2)}(\int_0^t\|\na u\|_{L^2}\,ds)^2\leq
C\|u_0\|_{L^2}^4(1+t).\eeno

Then we deduce from \eqref{2.1} to \eqref{2.3} that \beno
&&\f{d}{dt}\|\sqrt{\rho}u(t)\|_{L^2}^2 +
g^2(t)\|\sqrt{\rho}u(t)\|_{L^2}^2\\
&\lesssim& g^2(t)\|u_0\|_{L^p}^2(1+t)^{-2\b(p)} + g^6(t)\|\rho_0-1\|_{L^2}^2\|u_0\|_{L^2}^2(1+t)\\
&& + g^4(t)(\|\rho_0-1\|_{L^2}^2\|\na u_0\|_{L^2}^2+\|u_0\|_{L^2}^4)(1+t)\\
&\lesssim& g^2(t)\|u_0\|_{L^p}^2(1+t)^{-2\b(p)} +
g^4(t)(\|\rho_0-1\|_{L^2}^2\|u_0\|_{H^1}^2+\|u_0\|_{L^2}^4)(1+t).
\eeno Taking $g^2(t)=\frac{2}{(e+t)\ln(e+t)},$ then
$e^{\int_0^tg^2(s)\,ds}=\ln^2(t+e)$ and \beno
&&\ln^2(t+e)\|u(t)\|_{L^2}^2 \\
&\lesssim& \|u_0\|_{L^2}^2+ \int_0^t [\|u_0\|_{L^p}^2\frac{\ln(s+e)}{(s+e)^{1+2\b(p)}} + (\|\rho_0-1\|_{L^2}^2\|u_0\|_{H^1}^2+\|u_0\|_{L^2}^4)(s+e)^{-1}]\,ds\\
&\lesssim& \Bigl(\|u_0\|_{L^p}^2 +
(1+\|\rho_0-1\|_{L^2}^2)\|u_0\|_{H^1}^2+\|u_0\|_{L^2}^4
\Bigr)\ln(t+e), \eeno which gives \beq\label{2.4} \|u(t)\|_{L^2}^2
\lesssim \Bigl(\|u_0\|_{L^p}^2 +
(1+\|\rho_0-1\|_{L^2}^2)\|u_0\|_{H^1}^2+\|u_0\|_{L^2}^4
\Bigr)\ln^{-1}(t+e). \eeq Now we improve the estimate \eqref{2.4}.

We choose $f(t)=t+e$ in \eqref{ener1}, then we have \beno
\sup_{0<t<\infty}(t+e)\|\na u(t)\|_{L^2}^2 + \int_0^\infty
(t+e)\|u_t\|_{L^2}^2\,dt \leq C\|u_0\|_{H^1}^2
\exp\{C\|u_0\|_{L^2}^4\}, \eeno so that \beno
(\int_0^t\|\mathcal{F}_x\bigl((1-\rho)u_t
\bigr)\|_{L^\infty_{\xi}}\,ds)^2 &\leq&
\|1-\rho\|_{L^\infty_t(L^2)}^2(\int_0^t\|u_t\|_{L^2}\,ds)^2\\
&\leq& C\|\rho_0-1\|_{L^2}^2 \int_0^t (s+e)\|u_t\|_{L^2}^2\,ds
\int_0^t (s+e)^{-1}\,ds \\
&\leq& C\|\rho_0-1\|_{L^2}^2\|u_0\|_{H^1}^2 \exp\{C\|u_0\|_{L^2}^4\}
\ln(t+e). \eeno

\beno &&(\int_0^t\|\mathcal{F}_x\bigl(\rho u\na
u\bigr)\|_{L^\infty_{\xi}}\,ds)^2 \leq (\int_0^t
\|u(s)\|_{L^2}\|\na u(s)\|_{L^2}\,ds)^2 \\
&\leq& C\Bigl(\|u_0\|_{L^p}^2 +
(1+\|\rho_0-1\|_{L^2}^2)\|u_0\|_{H^1}^2 \Bigr)\|u_0\|_{H^1}^2
\exp\{C\|u_0\|_{L^2}^4\}(t+e)\ln^{-1}(t+e).\eeno We plug these
estimate into \eqref{2.2} and take $g^2(t)=\frac{3}{(e+t)\ln(e+t)},$
then $e^{\int_0^tg^2(s)\,ds}=\ln^3(t+e)$ and \beno
\ln^3(t+e)\|u(t)\|_{L^2}^2 &\lesssim& K\int_0^t [\f{\ln^2(s+e)}{(s+e)^{1+2\b(p)}} + \f{\ln^2(s+e)}{(s+e)^2} + (s+e)^{-1}]\,ds\\
&\lesssim& K\ln(t+e), \eeno which implies \beq\label{2.5}
\|u(t)\|_{L^2}^2 \lesssim K\ln^{-2}(t+e). \eeq So that \beqo
\int_0^\infty (t+e)^{-1} \|u\|_{L^2}^2\,dt \lesssim K. \eeqo We
choose $f'(t) = \ln (t+e)$ in \eqref{ener2}, then get \beno
&&\sup_{0<t<\infty} \ln (t+e) \|u(t)\|_{L^2}^2 + \int_0^\infty
\ln(t+e)\|\na u\|_{L^2}^2\,dt \\
&\leq& C\bigl(\|u_0\|_{L^2}^2 +\int_0^\infty (t+e)^{-1} \|u\|_{L^2}^2\,dt \bigr)\\
&\lesssim& K. \eeno Consequently, we take $f(t)=(t+e)\ln(t+e)$ in
\eqref{ener1}, obtain that \beno &&\sup_{0<t<\infty} (t+e)\ln (t+e)
\|\na u(t)\|_{L^2}^2 + \int_0^\infty (t+e)\ln(t+e)\|
u_t\|_{L^2}^2\,dt \\
&\leq& C(\|\na u_0\|_{L^2}^2 + \int_0^\infty
(\ln(t+e)+1)\|\na u\|_{L^2}^2\,dxdt)\exp\{C\|u_0\|_{L^2}^4\}\\
&\lesssim& K, \eeno which implies \beq\label{2.6} \|\na
u(t)\|_{L^2}^2\lesssim K(t+e)^{-1}\ln^{-1}(t+e). \eeq Combining
\eqref{2.5} and \eqref{2.6}, we get the revised estimates, \beno
(\int_0^t \|u\|_{L^2}\|\na u\|_{L^2}\,ds)^2 &\lesssim& K^2(\int_0^t
(s+e)^{-\f12}\ln^{-\f32}(s+e)\,ds)^2\\
&\lesssim& K^2(t+e)\ln^{-3}(t+e),\eeno \beno
(\int_0^t\|1-\rho\|_{L^2} \|u_t\|_{L^2}\,ds)^2 &\lesssim&K (\int_0^t
(s+e)\ln(s+e)
\|u_t\|^2_{L^2}\,ds)(\int_0^t (s+e)^{-1}\ln^{-1}(s+e)\,ds)\\
&\lesssim& K^2\ln\bigl(\ln(t+e)\bigr).\eeno Substituting these two
estimates in \eqref{2.2}, and taking $g^2(t) = \f5{(t+e)\ln(t+e)},$
then $e^{\int_0^tg^2(s)\,ds}=\ln^5(t+e)$ and \beno
\ln^5(t+e)\|u(t)\|_{L^2}^2 &\lesssim& \|u_0\|_{L^2}^2+ \int_0^t [\|u_0\|_{L^p}^2\f{\ln^4(s+e)}{(s+e)^{1+2\b(p)}} + K^2\f{\ln^3(s+e) \ln\bigl(\ln(t+e)\bigr)}{(s+e)^2} + \f{K^2}{s+e}]\,ds\\
&\lesssim& (K+K^2)\ln(t+e), \eeno from which, we obtain
\beq\label{2.7}\|u(t)\|_{L^2}^2\lesssim (K+K^2)\ln^{-4}(t+e).\eeq We
choose $f'(t)=\ln^2(t+e) $ in \eqref{ener2}, then\beno
&&\sup_{0<t<\infty} \ln^2 (t+e) \|u(t)\|_{L^2}^2 + \int_0^\infty
\ln^2(t+e)\|\na u\|_{L^2}^2\,dt\\ &\leq& C(\|u_0\|_{L^2}^2 +
\int_0^\infty
(t+e)^{-1}\ln(t+e)\|u(t)\|_{L^2}^2\,dt)\\
&\leq& C(\|u_0\|_{L^2}^2 + (K+K^2) \int_0^\infty
(t+e)^{-1}\ln^{-3}(t+e)\,dt)\\
&\lesssim& K+K^2.\eeno Finally, we take $f(t)=(t+e)\ln^2(t+e)$ in
\eqref{ener1} to get that \beno &&\sup_{0<t<\infty} (t+e)\ln^2 (t+e)
\|\na u(t)\|_{L^2}^2 + \int_0^\infty (t+e)\ln^2(t+e)\|
u_t\|_{L^2}^2\,dt \\
&\leq& C\Bigl(\|\na u_0\|_{L^2}^2 + \int_0^\infty
\bigl(\ln(t+e)+\ln^2(t+e)\bigr)\|\na u\|_{L^2}^2\,dt
\Bigr)\exp\{C\|u_0\|_{L^2}^4\}\\
&\lesssim&K+K^2.\eeno Consequently, we obtain \beno (\int_0^\infty
\|u_t\|_{L^2}\,dt)^2 \leq (\int_0^\infty
(t+e)\ln^2(t+e)\|u_t\|_{L^2}^2\,dt)(\int_0^\infty
(t+e)^{-1}\ln^{-2}(t+e)\,dt)\lesssim K+K^2. \eeno Which is the same
for $\P\dive\bigl(\mu(\rho)\cM(u)\bigr),
\Q\dive\bigl(\mu(\rho)\cM(u)\bigr)-\na\Pi \in L^1(\R^+; L^2),$ and
gives \eqref{decay2}.

Moreover \beno \int_0^\infty \|u\na u\|_{L^2}\,dt \leq \int_0^\infty
\bigl(\|u\|_{L^2}\|\na u\|_{L^2}^2+\|u_t\|_{L^2}\bigr)\,dt\lesssim
\sqrt{K}+K,\eeno and \beno (\int_0^t \|u\|_{L^2}\|\na
u\|_{L^2}\,ds)^2 &\leq& (\int_0^t
\ln^{-2}(s+e)\|u\|_{L^2}^2\,ds)(\int_0^t \ln^2(s+e)\|\na
u\|_{L^2}^2\,ds)\\&\lesssim& (K+K^2)\int_0^t
\ln^{-2}(s+e)\|u\|_{L^2}^2\,ds. \eeno Substituting these estimates
into \eqref{2.2}, noting that $2\b(p) \in (0,1),$ and taking
$g^2(t)=\frac{\al}{t+e}$ with any positive $\al\in(2\b(p),1),$ then
we get \beno
(t+e)^{\al}\|u(t)\|_{L^2}^2 &\lesssim& \|u_0\|_{L^2}^2+ (K+K^2)\int_0^t(s+e)^{\al-2}\int_0^s\ln^{-2}(\tau+e)\|u(\tau)\|_{L^2}^2\,d\tau ds \\
&+&K\int_0^t(s+e)^{\al-1-2\b(p)}\,ds + (K^2+K^3)\int_0^t(s+e)^{\al-2}\,ds \\
&\lesssim&(K+K^3)(t+e)^{\al-2\b(p)}+(K+K^2)
\int_0^t(s+e)^{\al-2}\int_0^s\ln^{-2}(\tau+e)\|u(\tau)\|_{L^2}^2\,d\tau
ds. \eeno For $t \geq 1,$ we define \beno
y(t) &\eqdefa&\int_{t-1}^t(s+e)^{\al}\|u(s)\|_{L^2}^2\,ds, \quad  \quad Y(t)\eqdefa\max\{y(s): 1\leq s \leq t\},\\
I(t) &\eqdefa&\int_0^t\ln^{-2}(s+e)\|u(s)\|_{L^2}^2\,ds. \eeno Then
recall that $\al<1,$ one has \ben\label{2.8}
I(t) &=& \int_0^{t-[t]}\ln^{-2}(s+e)\|u(s)\|_{L^2}^2\,ds + \int_{t-[t]}^t\ln^{-2}(s+e)\|u(s)\|_{L^2}^2\,ds \nonumber\\
&\lesssim& K + \sum_{j=0}^{[t]-1}\int_{t-j-1}^{t-j}\|u(s)\|_{L^2}^2(s+e)^{\al}(s+e)^{-\al}\ln^{-2}(s+e)\,ds \nonumber\\
&\lesssim& K + Y(t)\sum_{j=0}^{[t]-1}(t-j)^{-\al}\ln^{-2}(t-j)
\lesssim K + Y(t)(t+e)^{1-\al}\ln^{-2}(t+e), \een from which, we
infer that \beqo y(t)\lesssim (K+K^3)(t+e)^{\al-2\b(p)} +(K+K^2)
\int_0^t(s+e)^{-1}\ln^{-2}(s+e)Y(s)\,ds. \eeqo Then, applying
Gronwall's inequality \eqref{gronwall}, we have
\beq\label{2.9}\begin{split} Y(t)&\lesssim
(K+K^3)(t+e)^{\al-2\b(p)}\\ &\quad+
(K^2+K^5)\int_0^t(s+e)^{\al-2\b(p)-1}\ln^{-2}(s+e)\exp\{(K+K^2)\int_s^t(\tau+e)^{-1}\ln^{-2}(\tau+e)\,d\tau\}\,ds\\
&\lesssim (K+K^3)(t+e)^{\al-2\b(p)} +
(K^2+K^5)\exp(K^2)\int_0^t(s+e)^{\al-2\b(p)-1}\ln^{-2}(s+e)\,ds\\
&\lesssim (K+K^5)\exp(K^2)(t+e)^{\al-2\b(p)}.\end{split} \eeq
Plunging \eqref{2.9} into \eqref{2.8} gives rise to $I(t) \lesssim
(K+K^5)\exp(K^2)(t+e)^{1-2\b(p)}\ln^{-2}(t+e),$ we obtain \beno
(t+e)^{\al}\|u(t)\|_{L^2}^2 &\lesssim& (K+K^3)(t+e)^{\al-2\b(p)}+ (K^2+K^7)\exp(K^2)\int_0^t(s+e)^{\al-2\b(p)-1}\ln^{-2}(s+e)\,ds\\
&\lesssim& (K+K^7)\exp(K^2)(t+e)^{\al-2\b(p)}, \eeno which gives the
first inequality of \eqref{decay1}.

Go back to \eqref{ener2}, we choose $f''(t)$ such that
$\int_0^\infty f''(t)\|u(t)\|_{L^2}^2\,dt$ is finite. For example,
we let $f''(t)= (t+e)^{-1+2\b(p)-\ep}$ for any $\ep >0,$ (or
$f''(t)= (t+e)^{-1+2\b(p)}\ln^{-\al}(t+e)$ for any $\al>1,$) then
$f(t)= (t+e)^{1+2\b(p)-\ep}.$ Finally, we get \eqref{decay3} and the
second inequality of \eqref{decay1}.
\end{proof}

\setcounter{equation}{0}
\section{The Proof of Theorem \ref{mainthm2}}

The proof of Theorem \ref{mainthm2} is very similar to Theorem
\ref{mainthm1}. We should estimate every term in terms of
$\|u_0\|_{H^{\al}}$ instead of $\|u_0\|_{H^{1}}.$ First, we choose
$f(t)= t+e$ and $t$ in \eqref{ener1}, get that \beno
\sup_{0<t<\infty}(t+e)\|\na u(t)\|_{L^2}^2 + \int_0^\infty
(t+e)\|u_t\|_{L^2}^2\,dt \leq C\|u_0\|_{H^1}^2
\exp\{C(1+\|v\|_{L^\infty(L^2)}^2)\|\na v\|_{L^2(L^2)}^2\}, \eeno
and \beno \sup_{0<t<\infty}t\|\na u(t)\|_{L^2}^2 + \int_0^\infty
t\|u_t\|_{L^2}^2\,dt \leq C\|u_0\|_{L^2}^2
\exp\{C(1+\|v\|_{L^\infty(L^2)}^2)\|\na v\|_{L^2(L^2)}^2\}. \eeno By
interpolation, and let $v=u,$ we get that \beq\label{3.1}
\sup_{0<t<\infty}(t+e)^{\al}t^{1-\al}\|\na u(t)\|_{L^2}^2 +
\int_0^\infty (t+e)^{\al}t^{1-\al}\|u_t\|_{L^2}^2\,dt \leq
C\|u_0\|_{H^{\al}}^2 \exp\{C\|u_0\|_{L^2}^4\}. \eeq So that \beno
(\int_0^t\|\mathcal{F}_x\bigl((1-\rho)u_t
\bigr)\|_{L^\infty_{\xi}}\,ds)^2 &\leq&
\|1-\rho\|_{L^\infty_t(L^2)}^2(\int_0^t\|u_t\|_{L^2}\,ds)^2\\
&\leq& C\|\rho_0-1\|_{L^2}^2 \int_0^t
s^{1-\al}(s+e)^{\al}\|u_t\|_{L^2}^2\,ds
\int_0^t s^{\al-1}(s+e)^{-\al}\,ds \\
&\leq& C_{\al}t^{\al}, \eeno \beno
(\int_0^t\|\mathcal{F}_x\bigl((\mu(\rho)-\mu_0)\cM(u)\bigr)\|_{L^\infty_{\xi}}\,ds)^2&\leq&
(\int_0^t \|(\mu(\rho)-\mu_0)\cM(u)\|_{L^1} \,ds)^2\\
&\leq&\|\mu(\rho)-\mu_0\|_{L^\infty_t(L^2)}^2(\int_0^t\|\na
u\|_{L^2}\,ds)^2\\
&\leq& C\|\rho_0-1\|_{L^2}^2\|u_0\|_{L^2}^2(1+t), \eeno \beno
(\int_0^t\|\mathcal{F}_x\bigl(\rho u\na
u\bigr)\|_{L^\infty_{\xi}}\,ds)^2 \leq \|\rho
u\|^2_{L^\infty_t(L^2)}(\int_0^t\|\na u\|_{L^2}\,ds)^2\leq
C\|u_0\|_{L^2}^4(1+t).\eeno From which, we can deduce that \beno
&&\f{d}{dt}\|\sqrt{\rho}u(t)\|_{L^2}^2 +
g^2(t)\|\sqrt{\rho}u(t)\|_{L^2}^2\\
&\leq& C_{\al}\Bigl(g^2(t)(1+t)^{-2\b(p)} + g^6(t)(1+t)+ g^4(t)(1+t) + g^4(t) t^{\al}\Bigr)\\
&\leq& C_{\al}\Bigl(g^2(t)(1+t)^{-2\b(p)} + g^4(t)(1+t)+ g^4(t)
t^{\al}\Bigr). \eeno Taking $g^2(t)=\frac{2}{(e+t)\ln(e+t)},$ then
$e^{\int_0^tg^2(s)\,ds}=\ln^2(t+e)$ and \beno
&&\ln^2(t+e)\|u(t)\|_{L^2}^2 \\
&\leq& C\|u_0\|_{L^2}^2+ C_{\al}\int_0^t [\frac{\ln(s+e)}{(s+e)^{1+2\b(p)}} + \f1{s+e}+\f1{(s+e)^{2-\al}}]\,ds\\
&\leq& C_{\al}\ln(t+e), \eeno which gives \beq\label{3.2}
\|u(t)\|_{L^2}^2 \leq C_{\al}\ln^{-1}(t+e). \eeq Now, for $t>1,$ we
have \beno (\int_0^t\|\mathcal{F}_x\bigl(\rho u\na
u\bigr)\|_{L^\infty_{\xi}}\,ds)^2 &\leq& C (\int_0^t\|u\|_{L^2}\|\na u\|_{L^2}\,ds)^2\\
&\leq& C_{\al}(\int_0^t
s^{\f{\al-1}2}(s+e)^{-\f{\al}2}\ln^{-\f12}(s+e) \,ds)^2\\
&\leq&C_{\al}\bigl(1+(t+e)\ln^{-1}(t+e)\bigr)\leq
C_\al(t+e)\ln^{-1}(t+e).\eeno We take
$g^2(t)=\frac{3}{(e+t)\ln(e+t)},$ then
$e^{\int_0^tg^2(s)\,ds}=\ln^3(t+e)$ and \beno
\ln^3(t+e)\|u(t)\|_{L^2}^2 &\leq&C\|u_0\|_{L^2}^2+ C_{\al} \int_0^t [\f{\ln^2(s+e)}{(s+e)^{1+2\b(p)}} + \f{(s^{\al}+1)\ln(s+e)}{(s+e)^2} + \f1{s+e}]\,ds\\
&\leq& C_{\al}\ln(t+e), \eeno which implies \beqo \|u(t)\|_{L^2}^2
\leq C_{\al}\ln^{-2}(t+e), \qquad \hbox{for}~t>1.\eeqo And for
$0<t<1,$ it is obvious, so that \beq\label{3.3}\|u(t)\|_{L^2}^2 \leq
C_{\al}\ln^{-2}(t+e),\eeq and \beno \int_0^\infty (t+e)^{-1}
\|u\|_{L^2}^2\,dt \leq C_\al\int_0^\infty (t+e)^{-1}
\ln^{-2}(t+e)\,dt\leq C_\al. \eeno We choose $f'(t) = \ln (t+e)$ in
\eqref{ener2}, then get \beno &&\sup_{0<t<\infty} \ln (t+e)
\|u(t)\|_{L^2}^2 + \int_0^\infty
\ln(t+e)\|\na u\|_{L^2}^2\,dt \\
&\leq& C\bigl(\|u_0\|_{L^2}^2 +\int_0^\infty (t+e)^{-1} \|u\|_{L^2}^2\,dt \bigr)\\
&\leq& C_\al. \eeno Consequently, for any $0<r<\al,$ we take
$f(t)=t^{1-r}(t+e)^r\ln(t+e)$ in \eqref{ener1}, obtain that \beno
&&\sup_{0<t<\infty} t^{1-r}(t+e)^r\ln (t+e) \|\na u(t)\|_{L^2}^2 +
\int_0^\infty t^{1-r}(t+e)^r\ln(t+e)\|
u_t\|_{L^2}^2\,dt \\
&\leq& C \int_0^\infty
\bigl[(\f{t}{t+e})^{1-r}+\ln(t+e)\bigl((\f{t}{t+e})^{1-r}+
(\f{t+e}{t})^{r}\bigr)\bigr]\|\na
u\|_{L^2}^2\,dt\exp\{C\|u_0\|_{L^2}^4\}. \eeno Using \eqref{3.1}, we
get that \beqo \int_0^1 t^{-r}\|\na u(t)\|_{L^2}^2\,dt\leq C_\al
\int_0^1 t^{\al-r-1}(t+e)^{-\al}\,dt \leq C_\al, \eeqo which implies
\beq\label{3.4} \|\na u(t)\|_{L^2}^2\leq C_\al
t^{r-1}(t+e)^{-r}\ln^{-1}(t+e). \eeq Combining \eqref{3.3} and
\eqref{3.4}, for any $t>1,$ we get the revised estimates, \beno
&&(\int_0^t
\|u\|_{L^2}\|\na u\|_{L^2}\,ds)^2 \\
&\leq& C_\al(\int_0^1 s^{\f{r-1}2}
(s+e)^{-\f{r}2}\ln^{-\f12}(s+e)\,ds)^2+C_\al(\int_1^t s^{\f{r-1}2}
(s+e)^{-\f{r}2}\ln^{-\f32}(s+e)\,ds)^2\\
&\leq& C_\al\bigl(1+(t+e)\ln^{-3}(t+e)\bigr)\leq C_\al
(t+e)\ln^{-3}(t+e),\eeno \beno &&(\int_0^t\|1-\rho\|_{L^2}
\|u_t\|_{L^2}\,ds)^2 \\&\leq& C_\al(\int_0^t s^{1-r}(s+e)^r\ln(s+e)
\|u_t\|^2_{L^2}\,ds)(\int_0^t s^{r-1}(s+e)^{-r}\ln^{-1}(s+e)\,ds)\\
&\leq&C_\al \ln\bigl(\ln(t+e)\bigr).\eeno For $t>1,$ taking $g^2(t)
= \f5{(t+e)\ln(t+e)},$ then $e^{\int_0^tg^2(s)\,ds}=\ln^5(t+e)$ and
\beno
&&\ln^5(t+e)\|u(t)\|_{L^2}^2 \\
&\leq& C\|u_0\|_{L^2}^2+ C_\al\int_0^t [\f{\ln^4(s+e)}{(s+e)^{1+2\b(p)}} + \f{\ln^3(s+e) \ln\bigl(\ln(t+e)\bigr)}{(s+e)^2} + \f{1}{s+e}]\,ds\\
&\leq& C_\al\ln(t+e), \eeno from which, we obtain
\beq\label{3.5}\|u(t)\|_{L^2}^2\leq C_\al\ln^{-4}(t+e).\eeq We
choose $f'(t)=\ln^2(t+e) $ in \eqref{ener2}, then\beno
&&\sup_{0<t<\infty} \ln^2 (t+e) \|u(t)\|_{L^2}^2 + \int_0^\infty
\ln^2(t+e)\|\na u\|_{L^2}^2\,dt\\ &\leq& C(\|u_0\|_{L^2}^2 +
\int_0^\infty
(t+e)^{-1}\ln(t+e)\|u(t)\|_{L^2}^2\,dt)\\
&\leq& C_\al(1 +\int_0^1(t+e)^{-1}\,dt  +\int_1^\infty
(t+e)^{-1}\ln^{-3}(t+e)\,dt)\\
&\leq& C_\al.\eeno Finally, we take $f(t)=t^{1-r}(t+e)^r\ln^2(t+e)$
in \eqref{ener1} to get that \beno &&\sup_{0<t<\infty}
t^{1-r}(t+e)^r\ln^2 (t+e) \|\na u(t)\|_{L^2}^2 + \int_0^\infty
t^{1-r}(t+e)^r\ln^2(t+e)\|
u_t\|_{L^2}^2\,dt \\
&\leq& C \int_0^\infty
\bigl\{\ln(t+e)(\f{t}{t+e})^{1-r}+\ln^2(t+e)[(\f{t}{t+e})^{1-r}+(\f{t+e}{t})^{r}]\bigr\}\|\na
u\|_{L^2}^2\,dt
\exp\{C\|u_0\|_{L^2}^4\}\\
&\leq& C_\al.\eeno Consequently, we obtain \beno &&(\int_0^\infty
\|u_t\|_{L^2}\,dt)^2 \\
&\leq& (\int_0^\infty
t^{1-r}(t+e)^r\ln^2(t+e)\|u_t\|_{L^2}^2\,dt)(\int_0^\infty t^{r-1}
(t+e)^{-r}\ln^{-2}(t+e)\,dt)\\
&\leq& C_\al. \eeno Which is the same for
$\P\dive\bigl(\mu(\rho)\cM(u)\bigr),
\Q\dive\bigl(\mu(\rho)\cM(u)\bigr)-\na\Pi \in L^1(\R^+; L^2),$ and
gives \eqref{timedecay5}.

Then follow the same line to the proof of Theorem \ref{mainthm1}, we
get the first inequality of \eqref{timedecay4}. We choose
$f'(t)=(t+e)^{2\b-\ep}$ in \eqref{ener2}, obtain that \beno
&&\sup_{0<t<\infty} (t+e)^{2\b-\ep} \|u(t)\|_{L^2}^2 + \int_0^\infty
(t+e)^{2\b-\ep}\|\na u\|_{L^2}^2\,dt \\
&\leq& C\bigl(\|u_0\|_{L^2}^2 +\int_0^\infty (t+e)^{-1+2\b-\ep} \|u\|_{L^2}^2\,dt \bigr)\\
&\leq& C_\al. \eeno Then taking $f(t)=t^{1-r}(t+e)^{r+2\b-\ep}$ in
\eqref{ener1}, we deduce that \beno &&\sup_{0<t<\infty}
t^{1-r}(t+e)^{r+2\b-\ep} \|\na u(t)\|_{L^2}^2 + \int_0^\infty
t^{1-r}(t+e)^{r+2\b-\ep}\|
u_t\|_{L^2}^2\,dt \\
&\leq& C \int_0^\infty(t+e)^{2\b-\ep}
\bigl[(\f{t}{t+e})^{1-r}+(\f{t+e}{t})^{r}]\|\na u\|_{L^2}^2\,dt
\exp\{C\|u_0\|_{L^2}^4\}\\
&\leq& C_\al,\eeno which implies \eqref{timedecay6} and the second
inequality of \eqref{timedecay4}. This completes the proof of
Theorem \ref{mainthm2}.

\setcounter{equation}{0}
\section{Application: Global existence of \eqref{INS}}
First we present a polynomial relation between the velocity and the
density, which is the case between exponential and linear cases. In
general, we consider the case of non-Lipschitz velocity, the loss of
regularity will occur. With the non-Lipschitz velocity and
logarithms regular density, we have the following proposition.
\begin{prop}\label{tranprop}
{\sl For $\eta >0,$ let $\rho_0\in B^{(\eta+1) \ln}_{\infty,1}$ and
$\na u \in L^1(\R^+; B^0_{\infty,2}).$ Then we have $\rho\in
L^\infty((0,\infty); B^{\eta \ln}_{\infty,1}),$ and
\beq\label{rhoest} \|\rho(t)\|_{B^{\eta \ln}_{\infty,1}} \leq C
\|\rho_0\|_{B^{(\eta+1) \ln}_{\infty,1}} (\int_0^t\|\na
u\|_{B^0_{\infty, 2}}\,d\tau)^{\eta+1}, \eeq for any $t>0.$}
\end{prop}
\begin{proof}
First, we observe the continuity equation as follow: $\rho =
\sum_{j\geq -1} \rho_j,$ where $\rho_j$ satisfies
\beq\label{transport1} \left\{
\begin{array}
{l} \displaystyle \pa_t\rho_j + u\cdot\na \rho_j = 0, \\
\displaystyle \rho_j|_{t=0} = \D_j \rho_0.
\end{array}
\right. \eeq Then we have \beqo
\|\rho_j(t)\|_{B^{\f14}_{\infty,\infty}}\leq C
\|\rho_j(0)\|_{B^{\f12}_{\infty,\infty}} \exp\{ C\int_0^t\|\na
u\|_{B^0_{\infty 2}}\,d\tau \},\eeqo so that \beqo \|\D_q
\rho_j(t)\|_{L^\infty}\leq C
2^{-\f12(\f{q}2-j)}\|\D_j\rho_0\|_{L^\infty}F(u), \eeqo where
$F(u)=\exp\{ C\int_0^t\|\na u\|_{B^0_{\infty 2}}\,d\tau \}.$ If
$\f{q}2-j>N,$ for a positive inter number $N$ will be fixed later,
we obtain that \beno \sum_q \sum_{j<\f{q}2-N} \|\D_q
\rho_j(t)\|_{L^\infty} (2+q)^{\eta} &\leq& C F(u)\sum_q (2+q)^{\eta}
\sum_{j<\f{q}2-N} 2^{-\f12(\f{q}2-j)}\|\D_j\rho_0\|_{L^\infty}\\
&=& C F(u) \sum_j 2^{\f12 j}\|\D_j\rho_0\|_{L^\infty}
\sum_{q>2(N+j)}(2+q)^\eta 2^{-\f{q}4}\\
&\leq& C F(u)\sum_j
2^{\f12 j}\|\D_j\rho_0\|_{L^\infty}2^{-\f12(N+j)}(2+2(N+j))^\eta\\
&\leq& C F(u) 2^{-\f{N}2}\sum_j \|\D_j\rho_0\|_{L^\infty}(1+N+j)^\eta\\
&\leq& C 2^{-\f{N}2}N^\eta \|\rho_0\|_{B^{\eta \ln}_{\infty,1}}
F(u). \eeno On the other hand, we have \beqo
\|\rho_j(t)\|_{B^{-\f34}_{\infty,\infty}}\leq C
\|\rho_j(0)\|_{B^{-\f12}_{\infty,\infty}} F(u),\eeqo which implies
that \beqo \|\D_q \rho_j(t)\|_{L^\infty}\leq C
2^{-\f12(j-\f{3q}2)}\|\D_j\rho_0\|_{L^\infty}F(u). \eeqo If
$j-\f{3q}2>N,$ then we obtain \beno \sum_q \sum_{j>\f{3q}2+N} \|\D_q
\rho_j(t)\|_{L^\infty} (2+q)^{\eta} &\leq&  C F(u) \sum_j
2^{-\f12j}\|\D_j\rho_0\|_{L^\infty}
\sum_{q<\f23 (j-N)}(2+q)^\eta 2^{\f{3q}4}\\
&\leq& C F(u)\sum_j
2^{-\f12j}\|\D_j\rho_0\|_{L^\infty}2^{\f12(j-N)}(2+\f23(j-N))^\eta\\
&\leq& C 2^{-\f{N}2}N^\eta \|\rho_0\|_{B^{\eta \ln}_{\infty,1}}
F(u). \eeno Note that $\|\D_q \rho_j(t)\|_{L^\infty}\leq
\|\rho_j(t)\|_{L^\infty} \leq \|\D_j \rho_0\|_{L^\infty},$ for
$\f{q}2-N \leq j\leq \f{3q}2+N,$ we get that \beno \sum_q
\sum_{\f{q}2-N \leq j\leq \f{3q}2+N } \|\D_q \rho_j(t)\|_{L^\infty}
(2+q)^{\eta} &\leq& \sum_j
\|\D_j\rho_0\|_{L^\infty}\sum_{\f23(j-N)\leq q \leq
2(j+N)}(2+q)^\eta\\
&\leq&  C\sum_j \|\D_j\rho_0\|_{L^\infty} \f{(2+2(j+N))^{\eta+1}-(2+\f23(j-N))^{\eta+1}}{\eta +1}\\
&\leq& C N^{\eta+1} \|\rho_0\|_{B^{(\eta+1) \ln}_{\infty,1}}. \eeno
Finally, we obtain that \beno \|\rho(t)\|_{B^{\eta \ln}_{\infty,1}}
&\leq& \sum_q \sum_j (2+q)^{\eta}
\|\D_q \rho_j(t)\|_{L^\infty}\\
&\leq& CN^{\eta}2^{-\f{N}2}\|\rho_0\|_{B^{\eta \ln}_{\infty,1}} F(u)
+ C
N^{\eta+1}\|\rho_0\|_{B^{(\eta+1) \ln}_{\infty,1}}\\
&\leq& CN^{\eta+1}\|\rho_0\|_{B^{(\eta+1) \ln}_{\infty,1}} (1+
2^{-\f{N}2}F(u)), \eeno where we use $\|\rho_0\|_{B^{\eta
\ln}_{\infty,1}} \leq \|\rho_0\|_{B^{(\eta+1) \ln}_{\infty,1}}.$ We
choose $2^{-\f{N}2}F(u)\thicksim 1,$ i.e. $N\thicksim \int_0^t\|\na
u\|_{B^0_{\infty 2}}\,d\tau,$ then \beqo \|\rho(t)\|_{B^{\eta
\ln}_{\infty,1}} \leq C \|\rho_0\|_{B^{(\eta+1) \ln}_{\infty,1}}
(\int_0^t\|\na u\|_{B^0_{\infty 2}}\,d\tau)^{\eta+1}. \eeqo
\end{proof}

Now, we present the product law with logarithms Besov space and the
usual Besov space.
\begin{prop}\label{productlaw}
{\sl Let $\eta >1,$ and $a\in B^{\eta\ln}_{\infty,1},$ $b\in
B^0_{\infty,2}.$ Then $ab\in B^0_{\infty,2},$ and
\beq\label{product} \|ab\|_{B^0_{\infty,2}}\leq
C\|a\|_{B^{\eta\ln}_{\infty,1}}\|b\|_{B^0_{\infty,2}}. \eeq}
\end{prop}
\begin{proof}
We use Bony's decomposition that \beqo ab=T_a b+ T_b a + R(a,b).
\eeqo For the first term, we have \beno \|\D_j T_a b\|_{L^\infty}
&\leq& \sum_{|j-q|\leq N}\|S_{q-1}a\|_{L^\infty}\|\D_q
b\|_{L^\infty}\\
&\leq& \|b\|_{B^0_{\infty,2}}\sum_{|j-q|\leq N} c_{q,2}\sum_{k\leq
q-2}\|\D_k a\|_{L^\infty}\\
&\leq&\|a\|_{B^{\eta\ln}_{\infty,1}}\|b\|_{B^0_{\infty,2}}\sum_{|j-q|\leq
N} c_{q,2}\sum_{k\leq q-2} c_{k,1}(2+k)^{-\eta}\\
&\lesssim& c_{j,2}
\|a\|_{B^{\eta\ln}_{\infty,1}}\|b\|_{B^0_{\infty,2}}, \eeno where we
use $\eta>1.$

To deal with $T_b a,$ one has \beno \|\D_j T_b a\|_{L^\infty} &\leq&
\sum_{|j-q|\leq N}\|\D_q
a\|_{L^\infty}\sum_{-1\leq k\leq q-2}\|\D_k b\|_{L^\infty}\\
&\leq&
\|a\|_{B^{\eta\ln}_{\infty,1}}\|b\|_{B^0_{\infty,2}}\sum_{|j-q|\leq
N} (2+q)^{-\eta}\sqrt{q}\\
&\lesssim&\|a\|_{B^{\eta\ln}_{\infty,1}}\|b\|_{B^0_{\infty,2}}\sum_{|j-q|\leq
N} (2+q)^{-(\eta-\f12)}\\
&\lesssim& c_{j,2}
\|a\|_{B^{\eta\ln}_{\infty,1}}\|b\|_{B^0_{\infty,2}}, \eeno where we
use again $\eta>1$ so that $\sum_{q\geq-1}
(2+q)^{-(2\eta-1)}<\infty.$

For the last term, we obtain that \beno \sum_{j\geq -1}\|\D_j
R(a,b)\|_{L^\infty}^2 &\leq& \sum_{j\geq -1} (\sum_{q\geq j-N}
\|\D_q a\|_{L^\infty}\|\tilde{\D}_q b\|_{L^\infty})2\\
&\leq& \sum_{j\geq -1} (\sum_{q\geq j-N} (2+q)^{2\eta}\|\D_q
a\|_{L^\infty}^2)  (\sum_{q\geq j-N} (2+q)^{-2\eta}\|\D_q
b\|_{L^\infty}^2)\\
&\leq& \|a\|_{B^{\eta\ln}_{\infty,1}}^2 \sum_{j\geq -1}\sum_{q\geq
j-N} (2+q)^{-2\eta}\|\D_q b\|_{L^\infty}^2\\
&\leq& \|a\|_{B^{\eta\ln}_{\infty,1}}^2 \|b\|_{B^0_{\infty,2}}^2
\sum_{q\geq-1} (2+q)^{-2\eta}c_{q,2}^2\sum_{-1\leq j\leq q+N}1\\
&\lesssim& \|a\|_{B^{\eta\ln}_{\infty,1}}^2
\|b\|_{B^0_{\infty,2}}^2\sum_{q\geq-1} (2+q)^{-(2\eta-1)}\\
&\lesssim& \|a\|_{B^{\eta\ln}_{\infty,1}}^2
\|b\|_{B^0_{\infty,2}}^2.\eeno By summing up the above estimates, we
get \eqref{product}.
\end{proof}
Now we are at the position to proof Theorem \ref{mainthm3}.

\no\begin{proof} We rewrite the momentum equation of \eqref{INS} as
\beqo \mu_0\tri u = \P(\rho u_t +\rho u\na u )
-\P\dive\bigl((\mu(\rho)-\mu_0)\cM(u)\bigr),\eeqo from which, we get
\beqo \mu_0(I-S_0)\na u = \na
(-\tri)^{-1}\P\dive(I-S_0)\bigl((\mu(\rho)-\mu_0)\cM(u)\bigr) - \na
(-\tri)^{-1}\P(I-S_0)(\rho u_t +\rho u\na u). \eeqo Now we can
estimate $\na u$ in the norm of $L^1_t(B^0_{\infty,2}).$ Note that
\beqo \|\na (-\tri)^{-1}(I-S_0) f\|_{B^0_{\infty,2}}\leq
C\|f\|_{L^2}, \eeqo and recall \eqref{product}, we obtain that \beno
\mu_0\|(I-S_0)\na u\|_{L^1_t(B^0_{\infty,2})} &\leq&
C\|\mu(\rho)-\mu_0\|_{L^\infty_t(B^{\eta\ln}_{\infty,1})}\|\na
u\|_{L^1_t(B^0_{\infty,2})} \\
&&+ C(\|u_t\|_{L^1_t(L^2)}+ \|u\na
u\|_{L^1(L^2)})\\
&\leq&C\|\mu(\rho)-\mu_0\|_{L^\infty_t(B^{\eta\ln}_{\infty,1})}\Bigl(\|(I-S_0)\na
u\|_{L^1_t(B^0_{\infty,2})} +\|S_0\na
u\|_{L^1_t(L^\infty)} \Bigr)\\
&&+ C(\|u_t\|_{L^1_t(L^2)}+ \|u\na u\|_{L^1(L^2)}). \eeno Let $c_1$
be a small enough positive constant, which will be determined later
on, we denote \beq\label{lifetime} \bar{T}\eqdefa \sup\Big\{t;
\|\mu(\rho)-\mu_0\|_{L^\infty_t(B^{\eta\ln}_{\infty,1})} \leq
c_1\mu_0\Big\}. \eeq Then for any $t\leq \bar{T},$ the assumption
\eqref{smallassume1} holds and \beqo \mu_0\|\na
u\|_{L^1_t(B^0_{\infty,2})}\leq C\Bigl(\mu_0\|S_0\na
u\|_{L^1_t(L^\infty)}+\|u_t\|_{L^1_t(L^2)}+ \|u\na
u\|_{L^1(L^2)}\Bigr). \eeqo Note that $p<\f43,$ we can find some
positive $\ep$ such that $\f12+2\b(p)-2\ep>1.$ Then using
interpolation \eqref{gradu}, and decay estimates \eqref{decay1},
\eqref{decay3}, we obtain that \beno \|S_0\na u\|_{L^1_t(L^\infty)}
&\lesssim&\|\na
u\|_{L^1_t(L^4)}\lesssim (\int_0^t(s+e)^{\f12+2\b(p)-2\ep}\|\na u\|_{L^4}^2\,ds)^{\f12}(\int_0^t(s+e)^{-(\f12+2\b(p)-2\ep)}\,ds)^{\f12}\\
&\lesssim& (\int_0^t(s+e)^{\f12+2\b(p)-2\ep}\|\na u\|_{L^2}\|\rho
u_t+\rho u\na u\|_{L^2}\,ds)^{\f12}\\
&\lesssim&\Bigl((K^{\f12}+K^{\f72})\exp(K^2)\int_0^t(s+e)^{\b(p)-\ep}\|\rho
u_t+\rho u\na u\|_{L^2}\,ds\Bigr)^{\f12}\\
&\lesssim&\Bigl((K^{\f12}+K^{\f72})\exp(K^2)\bigl(\int_0^t(s+e)^{1+2\b(p)-\ep}\|\rho
u_t+\rho u\na
u\|_{L^2}^2\,ds\bigr)^{\f12}\bigl(\int_0^t(s+e)^{-1-\ep}\,ds\bigr)^{\f12}\Bigr)^{\f12}\\
&\lesssim&(K^{\f12}+K^{\f72})\exp(K^2).\eeno Combining
\eqref{decay2}, we get that \beqo \mu_0 \|\na
u\|_{L^1_t(B^0_{\infty,2})} \leq
C\bigl(\mu_0(K^{\f12}+K^{\f72})\exp(K^2)+K^{\f12} +K\bigr).\eeqo
Recall the definition of $K$ and $G(\rho_0,u_0),$ we deduce that
\beq\label{5.1} \mu_0\|\na u\|_{L^1_t(B^0_{\infty,2})} \leq
C(1+\mu_0)G(\rho_0,u_0)\exp\{\exp(C\|u_0\|_{L^2}^4)\}. \eeq Now,
substituting \eqref{5.1} into \eqref{rhoest} leads to
\beq\label{5.2}\begin{split}
\|\mu(\rho)-\mu_0\|_{L^\infty_t(B^{\eta\ln}_{\infty,1})} &\leq
C\|\mu(\rho_0)-\mu_0\|_{B^{(\eta+1)\ln}_{\infty,1}}\bigl(\|\na
u\|_{L^1_t(B^0_{\infty,2})}\bigr)^{\eta+1}\\
&\leq
C\|\mu(\rho_0)-\mu_0\|_{B^{(\eta+1)\ln}_{\infty,1}}\Bigl(\f{(1+\mu_0)G(\rho_0,u_0)}{\mu_0}\Bigr)^{\eta+1}\exp\big\{(\eta+1)\exp(C\|u_0\|_{L^2}^4)\big\}\\
&\leq \f{c_1}2 \mu_0,
\end{split}\eeq
as long as $C_0$ is sufficiently large and $c_0$  small enough in
\eqref{smallassume2}. This contradicts with \eqref{lifetime} and it
in turn shows that $\bar{T}=\infty.$ So the Theorem is proven.
\end{proof}

\end{document}